\documentclass[12pt]{article}

\usepackage{CJK,CJKnumb,amsmath}
\usepackage{amsfonts}
\usepackage{amsthm}
\usepackage{geometry}
\usepackage{mathrsfs}
\usepackage{amsmath}
\usepackage{amssymb}
\usepackage{amsfonts}
\usepackage[percent]{overpic}
\usepackage{bm}
\usepackage[all]{xy}
\usepackage{graphicx}
\usepackage{subfigure}
\usepackage{latexsym}
\usepackage[colorlinks, linkcolor=blue, anchorcolor=blue, citecolor=blue]{hyperref}
\usepackage{epigraph}
\usepackage{hyperref}
\usepackage{fancyhdr}
\usepackage{comment}

\usepackage{mathtools}
\mathtoolsset{showonlyrefs}

\numberwithin{equation}{section}
\usepackage{color}
\usepackage{orcidlink}

\newtheorem{theorem}{Theorem}

\newtheorem{que}{Question}

\newtheorem{proposition}{Proposition}[section]

\newtheorem{lemma}{Lemma}[section]
\theoremstyle{remark}

\newtheorem*{ack}{Acknowledgement}
\newtheorem*{AI}{AI Use Disclosure}

\def\re{\operatorname{Re}}
\def\Im{\operatorname{Im}}

\begin{document}

\title{Non-analyticity of hairs of exponential maps}
\author{Weiwei Cui, Jiaxing Huang and Lingrui Wang}
\date{\today}

\maketitle

\begin{abstract}
Let $f_{\lambda}(z)=\lambda e^z$, where $0<\lambda<1/e$. In 1984, Devaney and Krych showed that the Julia set consists of pairwise disjoint hairs. Viana further in 1988 proved that these hairs are $C^{\infty}$. In this paper we show that hairs are not analytic except for trivial ones. This solves a long-standing open question.
\end{abstract}

\section{Introduction and main results}

Let $f(z)=\lambda e^z$, where $0<\lambda<1/e$. The Fatou set of $f$ is defined as the set of points in the complex plane where the iterates of $f$ form a normal family. 
The complement of the Fatou set is called the Julia set of $f$, denoted by $J(f)$. Devaney and Krych proved that $J(f)$ consists of uncountably many pairwise disjoint curves connecting a finite point (called the endpoints) to $\infty$ \cite{devaney8}. The curves are called hairs or dynamic rays. These are analogues of external rays from polynomial dynamics. Hairs have played an important role in the study of the dynamical properties of exponential maps.

In 1988, Viana proved that hairs are smooth, and asked whether they are analytic or not \cite{viana88}; see also \cite[Question 7.19]{bergweiler26}. We solve this question by giving a negative answer.

We say that a hair is trivial if it is a preimage of the hair in the real axis.

\begin{theorem}\label{maintheorem}
Let $f(z)=\lambda e^z$ with $0<\lambda<1/e$. Then each hair is not analytic, except for the trivial ones.
\end{theorem}

Kisaka and Shishikura generalized the result of Viana to entire functions of the form $P(z)e^{Q(z)}$, where $P$ and $Q$ are polynomials; \cite{shishikura7}. Analyticity of hairs is unknown for these functions.

Besides the above mentioned results, existence of hairs in the Julia sets was established by Bara\'nski for disjoint type functions of finite order \cite{baranski2}. It is thus natural to see if these hairs could also have certain regularities. However, one cannot in this case expect a similar result as that of Viana. Comd\"uhr constructed a disjoint type function of finite order with nowhere differentiable hairs \cite{comduhr19}. In a recent joint work of the first author with Mart\'i-Pete, Pardo-Sim\'on and Rempe, we constructed hairs of arbitrary Hausdorff and packing dimensions (subject to trivial relations) \cite{cmpr}. This raised the following question.

\begin{que}
Is there an entire function with a Cantor bouquet Julia set whose hairs are analytic, except for trivial ones?
\end{que}

In general, the topology of components of Julia sets of disjoint type functions, which need not be hairs, can be very complicated. We refer to \cite{rempe9} for a thorough discussion.

\section{Some preliminaries}

Throughout the paper, we will only consider exponential maps $f(z)=\lambda e^z$ for $0<\lambda<1/e$. For simplicity, we put $f=f_{\lambda}$.

The map $f$ has an attracting fixed point $\alpha$ and a repelling fixed point $\beta$ on the real axis satisfying $0<\alpha<1<\beta$. Devaney and Krych proved that the Fatou set of $f$ coincides with the attracting basin of $\alpha$; see \cite{devaney8}. Denoted it by $\mathcal{A}(\alpha)$. The Julia set of $f$ has zero Lebesgue measure \cite{mcmullen11}. Thus, we have the following result.

\begin{lemma}\label{basin}\
\begin{itemize}
\item $\mathcal{A}(\alpha)\supset H:=\{z \in \mathbb{C}: \re z <\beta\}$;
\item $\mathcal{A}(\alpha)$ is dense in $\mathbb{C}$.
\end{itemize}
\end{lemma}
\begin{proof}
Both are well known. The first follows since, for every $z \in H$, 
\[|f(z)|=\lambda e^{\re z}< \lambda e^{\beta}=\beta.\]
So $f(H)\subset D(0,\beta)$. But for any $z\in D(0,\beta)$ we have
\[|f(z)|=\lambda e^{\re z}\leq \lambda e^{|z|}=f(|z|).\]
Thus
\[|f^{n}(z)|\leq f^{n}(|z|)\to\alpha\]
as $n \to \infty$. So we see that $H\subset\mathcal{A}(\alpha)$.

The second is also true by the result of McMullen mentioned above.
\end{proof}

For $k \in \mathbb{Z}$, define the horizontal strip
\[P_k=\left\{z \in \mathbb{C}: -\pi+2k\pi < \Im(z) \le \pi+2k\pi\right\}.\]
So, $f$ maps the interior of each $P_k$ onto $\mathbb{C} \setminus (-\infty,0]$ conformally.

By Lemma \ref{basin}, for every $z \in J(f)$, the iterates $f^n(z)$, $n \ge 0$, stay in the right half plane $\{z \in \mathbb{C}: \re z \ge b\}$. Thus for every $n \ge 0$, there is a unique integer $s_n$ such that $f^n(z) \in P_{s_n}$, and 
\begin{equation}\label{eq: imaginary part}
    \left|\Im f^n(z)-2s_n\pi\right|<\frac{\pi}{2}.
\end{equation} 

Therefore, every $z \in J(f)$ corresponds to a sequence $\underline{s}(z)=(s_0, s_1, \ldots) \in \mathbb{Z}^{\mathbb{N}}$, which is called the external address of $z$. On the contrary, every sequence $\underline{s}$ defines a set 
\[J_{\underline{s}}=\left\{z \in J(f): \underline{s}(z)=\underline{s}\right\},\]
which is a hair (if non-empty).

A sequence $\underline{s}=(s_0, s_1, \ldots)$ is called exponentially bounded, if there is $x>0$ such that 
$$2\pi |s_n| \le f^n(x)$$
for every $n \ge 0$. 

\begin{lemma}\label{lm: non-empty iff exponentially bounded }
    $J_{\underline{s}}$ is non-empty if and only if $\underline{s}$ is exponentially bounded.
\end{lemma}
\begin{proof}
This is also well known. See, for instance, \cite{devaney8} and also \cite{schleicher4}.
\end{proof}

Fix an exponentially bounded address $\underline{s}=(s_0, s_1, \ldots)$. Let $z,w \in J_{\underline{s}}$ be distinct. Put 
\[z_n=f^n(z),\quad w_n=f^n(w)\]
and
\[d_n=\re \; z_n-\re \; w_n.\] 
\begin{lemma}\label{lm: real part goes to infinity}
    $|d_n| \to \infty$ and $\mathrm{sgn} \; d_n$ is eventually constant. 
\end{lemma}
\begin{proof}
    Consider $f^n: J(f) \to J(f) \subset \{z \in \mathbb{C}: \re z \ge \beta\}$. Since $P_{s_n}$ has width $2\pi$, the map $f: P_{s_{n}} \cap J(f) \to \{z \in \mathbb{C}: \re \; z \ge \beta\}$ is injective. Let $\gamma: [0,1] \to J(f)$ be an arbitrary rectifiable Jordan curve connecting $z$, $w$, with $\gamma(0)=z$, $\gamma(1)=w$, which is mapped to the segment $[z_n, w_n]$ by $f^n$. (Note that this is possible since Viana proved that hairs are $C^{\infty}$.) Denote by $\ell(\gamma)$ the length of $\gamma$. Then 
\[|z_n-w_n|=\int_{\gamma} |(f^n)'(\zeta)||d\zeta| \ge \beta^n \ell(\gamma) \ge \beta^n|z-w|,\]
since $J(f) \subset \{z \in \mathbb{C}: \re z \ge \beta\}$ by Lemma \ref{basin}. By \eqref{eq: imaginary part}, $|\Im(z_n-w_n)|<\pi$. Thus $|d_n| \to \infty$. 

Now let $n$ be sufficiently large so that $d_n>L$, where $L=\log (1+\pi/\beta)$. Then $|z_{n+1}|/|w_{n+1}| =e^{d_n} >1+\pi/\beta$. Suppose by contradiction that $d_{n+1} \le 0$, i.e. $\re z_{n+1} \le \re w_{n+1}$. Then $|z_{n+1}| \le |w_{n+1}|+\pi$. This implies that
\[\frac{|z_{n+1}|}{|w_{n+1}|}\leq 1+\frac{\pi}{|w_{n+1}|} \leq 1+\frac{\pi}{\beta},\]
a contradiction. This argument also works for $d_n<0$.
\end{proof}

By Lemma \ref{lm: real part goes to infinity}, without loss of generality, we may assume that $d_n \to \infty$ as $n \to \infty$. Then for sufficiently large $n$, using Lemma \ref{basin} we obtain
\begin{equation}\label{eq: real part larger than modulus}
    \begin{split}
        \frac{\re z_{n+1}}{|w_{n+1}|} & = \frac{\sqrt{|z_{n+1}|^2-(\Im  z_{n+1})^2}}{|w_{n+1}|} \\
        &\ge \sqrt{e^{2d_n}-\frac{(|\Im w_{n+1}|+\pi)^2}{|w_{n+1}|^2}} \\
        & \ge \sqrt{e^{2d_n}-(1+\pi/\beta)^2}\\
        & \to +\infty 
    \end{split}
\end{equation}
as $n\to\infty$.

\section{Proof of Theorem \ref{maintheorem}}
We prove Theorem \ref{maintheorem} by contradiction. Suppose that $J_{\underline{s}} \neq \emptyset$ is non-trivial and that there is a real analytic map
\[\phi:(-\epsilon, \epsilon) \to J_{\underline{s}}.\]
Let $z,w$ be two distinct points in $\phi((-\epsilon,\epsilon))$ and put
$$z_n=f^n(z),\quad w_n=f^n(w),$$
$$d_n=\re \; z_n-\re \; w_n.$$
Since $f'\neq 0$, $f^n\circ\phi$ is a real analytic map into $f^n(J_{\underline{s}})$, and by Lemma \ref{lm: real part goes to infinity}, we may assume that $\beta< |w|+2\pi <\re z-2\pi $, and $d_n \to +\infty$.

Put $\xi=|w|+2\pi$, $\eta=\re z-2\pi$. Furthermore, $z, w$ can be chosen such that $\xi<\eta<f(\eta)$ and $|z|<f(\eta)$.

Recall that the external address $\underline{s}$ is exponentially bounded. 

\begin{lemma}\label{lm: estimate on the address}
    For $n \ge 0$, we have $(2|s_n|+1)\pi \le f^n(\xi)$. 
\end{lemma}
\begin{proof}
    For $n \ge 0$, 
   \[ \begin{aligned}
    (2|s_n|+1)\pi &\le |\Im w_n|+3\pi/2\\
     &\le |w_n|+3\pi/2 \\
     &\le f^n(|w|) +3\pi /2 \\
     &\le f^n(\xi).\qedhere
    \end{aligned}\]
\end{proof}

Without loss of generality, we may assume that $z=\phi(0)$, $w=\phi(-\rho)$ for some $0<\rho <\epsilon$. Since $\phi$ is real analytic, it can be extended analytically to the disk $D(0,\epsilon)$.  Choose $M$ with $\max\{\xi, |z|\}<M<f(\eta)$ and $r\in (0,\epsilon)$ such that
\begin{equation}\label{1}
  \sup_{\zeta \in D(0,r)} |\phi(\zeta)|<M, \quad \re \phi(t)>\eta+1 \,\text{ for }\, |t|<r.  
\end{equation}

We claim that the estimate on the real part in \eqref{1} holds for all iterates.
\begin{lemma}\label{lm: real part of n-th iterate on the curve}
For $n \ge 0$ and $t \in (-r,r)$, we have $\re f^n(\phi(t))>f^n(\eta)+1$. 
\end{lemma}
\begin{proof}
    We prove by induction. For $n=0$, it is true by \eqref{1}. Now suppose that $\re f^k(\phi(t)) >f^k(\eta)+1$ holds for $k\leq n$. Then
\[|f^{n+1}(\phi(t))|=\lambda e^{\re f^n(\phi(t))}>e f^{n+1}(\eta).\]
By Lemma \ref{lm: estimate on the address} and the choice of $\xi$ and $\eta$, $|\Im f^{n+1}(\phi(t))| \le (2|s_{n+1}|+1) \pi \le f^{n+1}(\xi)<f^{n+1}(\eta)$. Thus 
$$\re f^{n+1}(\phi(t)) >\sqrt{e^2-1} \,f^{n+1}(\eta)>f^{n+1}(\eta)+1. \qedhere$$
\end{proof}

\medskip

We shall extend the estimates above to all points in $D(0,r)$. Define 
\[\psi(\zeta)=\overline{\phi(\overline{\zeta})} \quad\text{for}\quad\zeta \in D(0,r),\] which also satisfies the inequalities \eqref{1}. Put 
\[\Phi_n(\zeta):=f^n(\phi(\zeta))-f^n(\psi(\zeta))\]
and
\[\Psi_n(\zeta):=\Phi_n(\zeta)-4s_n\pi i.\]

\medskip

\begin{lemma}\label{lm: imaginary part close to a horizontal line}
For $0<r'<r$ and $\zeta \in D(0,r) \setminus [-r',r']$, we have
\[ |\Psi_n(\zeta)| \to 0\quad\text{as}\quad n \to \infty.\]
\end{lemma}
\begin{proof}
    Note that $\Im f^n (\phi(t)) -2s_n\pi=\mathrm{arg} \; f^{n+1}(\phi(t))$, where $\mathrm{arg}$ is the principal branch of the argument. Therefore, by Lemmas \ref{lm: estimate on the address} and \ref{lm: real part of n-th iterate on the curve}, we have 
\begin{equation}\label{eq: estimate imaginary part of n-th iterate of the curve}
    |\Im f^n(\phi(t))-2s_n \pi| \le \frac{|\Im f^{n+1}(\phi(t))|}{|\re f^{n+1}(\phi(t))|} \le \frac{f^{n+1}(\xi)}{f^{n+1}(\eta)}.
\end{equation}
 Since $\lambda \in (0,1/e)$, $f(\overline{z})=\lambda e^{\overline{z}}=\overline{f(z)}$. Applying the same argument to $f^n \circ \psi$, we have, for $t \in (-r,r)$, 
 \begin{equation}\label{eq: estimate imaginary part of n-th iterate of the reflection curve}
    |\Im f^n(\psi(t))-2s_n \pi| \le \frac{f^{n+1}(\xi)}{f^{n+1}(\eta)}.
\end{equation}

Thus by \eqref{eq: estimate imaginary part of n-th iterate of the curve} and \eqref{eq: estimate imaginary part of n-th iterate of the reflection curve}, $|\Psi_n(t)| \le 2 e^{f^n(\xi)-f^n(\eta)}$. While on $D(0,r)$, $|\phi(\zeta)|=|\psi(\zeta)|<M$. Thus by Lemma \ref{lm: estimate on the address}, since $\xi<M$, 
$$|\Psi_n(\zeta)| \le 2f^n(M)+4\pi |s_n| \le 2f^n(M)+2f^n(\xi)<4f^n(M). $$

Define $a_n=2 e^{f^n(\xi)-f^n(\eta)}$, $b_n=4 f^n(M)$. Fix $r' \in (0,r)$ and let $I=[-r',r']$ and $D=D(0,r)$. Fix $\zeta \in D\setminus I$ and let $\omega=\omega(\zeta, I, D)$ be the harmonic measure of $I$ in $D \setminus I$ viewed from $\zeta$. By the two-constants theorem \cite[pp. 41]{nevanlinna}, for the harmonic function $\log |\Psi_n(\zeta)|$ in $D \setminus I$, we have 
\begin{equation}\label{eq: estimate on the disk}
    \log |\Psi_n(\zeta)| \leq \omega\log a_n+(1-\omega) \log b_n. 
\end{equation}
Thus, for sufficiently large $n$, 
\begin{equation*}
    \begin{split}
        \log |\Psi_n(\zeta)| & \le \omega(\log 2+f^n(\xi)-f^n(\eta))+(1-\omega) (\log 4+\log f^n(M)) \\
        & = \omega f^{n}(\xi)-\omega f^{n}(\eta)+(1-\omega)f^{n-1}(M) \\
        & \quad +(1-\omega) \log \lambda +(2-\omega) \log 2 \\
        & =-\omega f^n(\eta) \left(1+\frac{\omega f^n(\xi)+(1-\omega)f^{n-1}(M)}{-\omega f^n(\eta)}\right).
    \end{split}
\end{equation*}
Since $\xi<\eta$, $M<f(\eta)$, we have, for sufficiently large $n$, 
\begin{equation*}
    \log |\Psi_n(\zeta)| \le \frac{-\omega}{2} f^n(\eta) \to -\infty \quad \text{as} \quad n \to \infty.
\end{equation*}
This implies that, for $\zeta\in D\setminus I$, $|\Psi_n(\zeta)| \to 0$ as $n \to \infty$. 
\end{proof}

By Lemma \ref{basin}, $\mathcal{A}(\alpha)$ is dense in $\mathbb{C}$. Since $\phi$ and $\psi$ are holomorphic on $D(0,r)$, $\phi^{-1}(\mathcal{A}(\alpha))$ and $\psi^{-1}(\mathcal{A}(\alpha))$ are both open and dense in $D(0,r)$. Hence we can choose $\zeta \in D(0,r) \cap \phi^{-1}(\mathcal{A}(\alpha)) \cap \psi^{-1}(\mathcal{A}(\alpha))$. For such $\zeta$, we have both $f^n(\phi(\zeta))$ and $f^n(\psi(\zeta))$ converge to the attracting fixed point $\alpha$. Hence $|\Phi_n(\zeta)| \to 0$ as $n \to \infty$. This implies that $4s_n\pi i=\Phi_n(\zeta)-\Psi_n(\zeta) \to 0$. Since $s_n$ is an integer, it has to be identically zero for all sufficiently large $n$. 

To sum up, we obtain the following result.

\begin{proposition}
Suppose that $J_{\underline{s}} \neq \emptyset$ and there is a real analytic mapping $\phi:(-\epsilon, \epsilon) \to J_{\underline{s}}$. Then there is an integer $N \ge 0$ such that $s_n=0$ for $n \ge N$. 
\end{proposition}

Since it is assumed from the beginning that $J_{\underline{s}}$ is non-trivial, we immediately reach a contradiction. This completes the proof of Theorem \ref{maintheorem}.

\medskip

\begin{ack}
Cui was partially supported by NSFC (No. 12401105, 12522108) and Shandong Provincial Natural Science Fund for Excellent Young Scientists Program (Overseas) (No. 2025HWYQ-021). Huang was partially supported by the NSFC (No.12201420 and 12231013) and the NSF of Guangdong Province (No. 2026A1515010851).  Wang was supported by Shandong Postdoctoral Science Foundation (Grant No.\ SDZZ-ZR-202501290).
\end{ack}

\smallskip

\begin{AI}
The authors developed the general strategy and idea of this paper. AI tools were used to refine some estimates. The authors take full responsibility for the content and correctness of the paper.
\end{AI}


\bigskip

\noindent {\bf Weiwei Cui and Lingrui Wang}\\
 Research Center for Mathematics and Interdisciplinary Sciences\\
Shandong University, Qingdao, 266237, China.

\smallskip

\noindent weiwei.cui@sdu.edu.cn, lrwang@sdu.edu.cn

\bigskip

\noindent {\bf Jiaxing Huang}\\
School of Mathematical Sciences\\
 Shenzhen University, Guangdong, 518060, China

\smallskip

\noindent hjxmath@szu.edu.cn

\bigskip

\end{document}